\documentclass[12pt]{amsart}
\usepackage{amssymb,amsmath,epsfig,graphics,latexsym,psfrag}
\usepackage[english]{babel}
\usepackage[all]{xy}
\usepackage{amscd, amssymb, amsmath, amsthm, graphics}
\usepackage{amsmath,amsfonts,amsthm,amssymb}
\usepackage{latexsym,amsmath}
\usepackage{graphicx,psfrag}
\usepackage{mathrsfs}

\newtheorem{theorem}{Theorem}[section]
\newtheorem{proposition}[theorem]{Proposition}

\newtheorem{lemma}[theorem]{Lemma}

\theoremstyle{definition}

\theoremstyle{remark}
\newtheorem{remark}[theorem]{Remark}

\newtheorem{question}{Question}
\numberwithin{equation}{section}
\renewcommand{\t}{ \widetilde}

\newcommand{\Z}{\mathbb Z}

\newcommand{\ZZ}{{\mathbb Z}}

\newcommand{\abs}[1]{\lvert#1\rvert}

\newcommand{\co}{\colon\thinspace}
\renewcommand{\epsilon}{\varepsilon}
\renewcommand{\c}{\mathcal}

\usepackage{times}

\begin{document}
\sloppy

\title[]{Finiteness of mapping degree sets for 3-manifolds}

\author{Pierre Derbez}
\address{Centre de Math\'ematiques et d'Informatique,
Technopole de Chateau-Gombert,
39, rue Fr\'ed\'eric Joliot-Curie -
 13453 Marseille Cedex 13}
\email{derbez@cmi.univ-mrs.fr}

\author{Hongbin Sun}
\address{Department of Mathematics \\ Princeton University \\ Princeton NJ 08544 USA}
\email{hongbins@math.princeton.edu}

\author{Shicheng Wang}
\address{Department of Mathematics, Peking University, Beijing, China}
\email{wangsc@math.pku.edu.cn}



\date{\today}
\begin{abstract}
By constructing certain maps, this note completes the answer of the
Question: For which closed orientable $3$-manifold  $N$,
 the set of mapping degrees $\c{D}(M,N)$ is finite for any closed orientable $3$-manifold $M$?
\end{abstract}


\maketitle

\vspace{-.5cm} 
\section{introduction}
 Let $M$ and $N$ be two closed oriented $3$-dimensional
manifolds. Let $\c{D}(M,N)$ be the set of degrees of maps from $M$
to $N$, that is
$$\c{D}(M,N)=\{d\in{\Z}\,| f\co M\to N,\, \, \,{\deg}(f)=d\}.$$
We  will  simply use $\c{D}(N)$ to denote $\c{D}(N,N)$, the set of
self-mapping degrees of $N$.

The calculation of $\c{D}(M,N)$ is a classical topic appeared in
many literatures. According to \cite{CT}, Gromov thought it is a
fundamental problem in topology to determine the set $\c{D}(M,N)$
for any dimension $n$.

The result is simple and well-known for dimension $n=1,2$.  For
dimension $n>3$, there are some interesting special results (See
\cite{DW} for recent ones and references therein), but it is
difficult to get general results, since there are no classification
results for manifolds of dimension $n>3$.

The case of dimension 3 becomes the most attractive in this topic.
Since Thurston's geometrization conjecture, which has been
confirmed, implies that closed orientable 3-manifolds can be
classified in reasonable sense.

A basic property of $\c{D}(M,N)$ is reflected  in the following:

\begin{question}\label{gromov} (see  also \cite[Problem A]{Re} and  \cite[Question 1.3]{W2}):
For which closed orientable $3$-manifolds $N$,  the set $\c{D}(M,N)$
is finite for any given closed oriented  $3$-manifold $M$?
\end{question}

The main result proved in this note is the following

\begin{theorem}\label{$D(M,N)$}
Let $N$ be a given closed oriented  3-manifold $N$. If
$|\c{D}(R)|=\infty$ for each prime factor $R$ of $N$, then there is
a closed orientable 3-manifold $M$ such that $|\c{D}(M,N)|=\infty$.
\end{theorem}

Theorem \ref{$D(M,N)$} follows from  an explicit result Theorem
\ref{$D(M,N)$*}, which provides the concrete $M$ and the infinite
set in $\c{D}(M,N)$ for the given $N$. The proof of Theorem
\ref{$D(M,N)$} (\ref{$D(M,N)$*}) is essentially  elementary, which
does not appear until now mainly due to two reasons:

(1) $|\c{D}(N)|$ may be finite even $|\c{D}(R)|=\infty$ for each
prime factor $R$ of $N$; for example $|\c{D}(T^3)|=\infty$ but
$|\c{D}(T^3\#T^3)|<\infty$ for 3-dimensional torus $T^3$ \cite{W1}.
Such phenomenon puzzled us to wonder if Theorem \ref{$D(M,N)$} was
always to be true \cite[page 460]{W2}.

(2) The target concerned in Theorem \ref{$D(M,N)$} became the only
unknown case for Question \ref{gromov}  just very recently. Now
Theorem \ref{$D(M,N)$} completes the answer of Question \ref{gromov}
and we have

\begin{theorem}\label{main} Let $N$ be a  closed orientable
3-manifold. Then there is  a  closed orientable 3-manifold $M$ such
that $|\c{D}(M,N)|=\infty$ if and only if  $|\c{D}(R)|=\infty$  for
each prime factor $R$ of $N$.
\end{theorem}

In the following we will make a brief recall of the development of
Theorem \ref{main}. To be able to do this we need to have a brief
look of today's picture of 3-manifolds.

\vskip 0.4 truecm

{\bf The picture of 3-manifolds:} Each closed orientable 3-manifold
$N$ has unique prime decomposition $N_1\#.....\#N_k$, the prime
factors are unique up to the order and up to homeomorphisms. Each
closed orientable prime 3-manifold $N$ has a unique geometric
decomposition  such that each geometric piece supports one of the
following eight geometries: $H^3$, $\widetilde{PSL}(2,R)$,
$H^2\times E^1$, Sol, Nil, $E^3$, $S^3$ and $S^2\times E^1$ (where
$H^n$, $E^n$ and $S^n$ are n-dimensional hyperbolic space, Euclidean
space and sphere respectively), for details see \cite{Th} and
\cite{Sc}. Moreover each geometric piece of $N$ with non-trivial
geometric decomposition supports either $H^3$-geometry or $H^2\times
E^1$-geometry, hence each 3-manifold supporting one of the remaining
six geometry is closed. Furthermore each 3-manifold supporting
geometries of either $H^2\times E^1$, or $E^3$, or $S^2\times E^1$
is covered by a trivial circle bundle, and each 3-manifold
supporting geometries of either Sol, or  Nil, or $E^3$ is covered by
a torus bundle. Call prime closed orientable 3-manifold $N$ a {\it
non-trivial graph manifold} if $N$ has non-trivial geometric
decomposition but contains no hyperbolic piece.

\vskip 0.4 truecm

{\bf The development of Theorem \ref{main}:} It is a common sense
for many people that $|D(N)|=\infty$ for 3-manifold $N$ which is
either a product of a surface and the circle, or $N$ is covered by
the 3-sphere.
 The first significant result in this direction
is due to Milnor and Thurston in the later 1970's. By using the
minimum integer number of 3-simplices to build $N$ \cite[Theorem
2]{MT}, they proved

\begin{theorem}\label{Milnor-Thurston}
For each given hyperbolic 3-manifold $N$, $|\c{D}(M,N)|<\infty$ for
any $M$.
\end{theorem}

Gromov \cite{G} introduced the simplicial volume $\|N\|$ for a
manifold $N$, which  is approximately the minimum real number of
3-simplices to build $N$. Gromov and Thurston proved that $\|N\|$
 is proportional to  the hyperbolic
volume of $N$ in the case of $N$ is a hyperbolic 3-manifold, and
then Soma proved  $\|N\|$ is proportional to the sum of the
hyperbolic volume of the hyperbolic pieces in the geometric
decomposition of $N$ (see \cite{G}, \cite{Th}, \cite{So}). $||*||$
respects  the mapping degrees, i.e.
 for any map $f\co M\to N$ then
 $||M||\geq\abs{{\rm deg}(f)}\cdot ||N||$. Then it is deduced that

\begin{theorem}\label{simplicial volume}
Suppose $N$ is a closed orientable 3-manifold. If a prime factor of
$N$  having hyperbolic piece in its geometric decomposition, then
$|\c{D}(M,N)|<\infty$ for any $M$.
\end{theorem}

 Brooks and Goldman   \cite{BG1} \cite{BG2} introduced
 the Seifert volume $SV(*)$ for closed orientable 3-manifolds which also respects the mapping degrees and is non-zero for
 each
3-manifold supporting the $\widetilde{PSL}(2, R)$ geometry. Then it
is deduced that

\begin{theorem}\label{Seifert volume}
Suppose $N$ is a closed orientable 3-manifold. If a prime factor of
$N$  supporting $\widetilde{PSL}(2, R)$ geometry. Then
$|\c{D}(M,N)|<\infty$ for any $M$.
\end{theorem}

Both Theorems \ref{simplicial volume} and \ref{Seifert volume} were
already known in the early 1980's. The following result is known no
later than early 1990's (see \cite{W1} for example).

\begin{proposition}\label{D(N)}
Suppose $N$ is a closed orientable 3-manifold. Then
$|\c{D}(N)|=\infty$
 if and only if either $N$ is
covered by a torus bundle or a trivial circle bundle, or each prime
factor of $N$ is covered by $S^3$ or $S^2\times E^1$.
\end{proposition}

After Theorems \ref{simplicial volume} \ref{Seifert volume} and
Proposition
 \ref{D(N)}, the remaining unknown cases for Question
 \ref{gromov} are:  either $N$ is a non-trivial graph manifold;  or  $N$ is  a non-prime 3-manifold,  and $|\c{D}(R)|=\infty$ for each prime
 factor $R$ of $N$,  but some $R$ is not covered by either $S^3$ or $S^2\times E^1$.

 In 2009  it is  proved in \cite{DeW} that each closed
orientable non-trivial graph manifold $N$ has a finite covering
$\t{N}$ with positive Seifert volume (it is still unknown weather
$SV(\tilde N)>0$ implies $SV(N)>0$ for a finite cover $\tilde N\to
N)$), and therefore it is deduced that

\begin{theorem}\label{virtully positive Seifert volume} Let $N$ be  closed orientable
non-trivial graph manifold. Then $|\c{D}(M,N)|< \infty$ for any
closed orientable 3-manifold $M$.
\end{theorem}

Theorems \ref{simplicial volume}  \ref{Seifert volume},
\ref{virtully positive Seifert volume} and \ref{$D(M,N)$} (and
Proposition \ref{D(N)}) imply Theorem \ref{main}.

\begin{remark} Recently $\c{D}(N)$ is completely determined for each $N$ with $|\c{D}(N)|=\infty$
(\cite{Du},
 \cite{SWW}, \cite{SWWZ}), which is useful in the proof of
 Theorem \ref{$D(M,N)$} (\ref{$D(M,N)$*}).
\end{remark}


\section{Proof of Theorem \ref{$D(M,N)$}}

Call a map $f: M\to N$ between connected manifolds is
$\pi_1$-surjective if the induced $f_*:\pi_1M\to \pi_1N$ is
surjective. We start with the following classical  fact in topology,
whose proof is inspired by Stallings's elegant proof of Grushko's
theorem \cite{St} and appeared in several papers (for an easy and
recent one, see \cite{RW}).

\begin{lemma}\label{basic}
Let $f\co M\to N$ be a $\pi_1$-surjective nonzero degree map between
closed oriented  $n$-manifolds, with $n\geq 3$. Then for any
$n$-ball $B$ in $N$, there exists a map $g$ homotopic to $f$ such
that $g^{-1}(B)$ is an $n$-ball in $M$.
\end{lemma}

Denote the subset of $\c{D}(M, N)$ which realized by
$\pi_1$-surjective map $f:M\to N$ as $\c{D}_{surj}(M,N)$. Then the
fact below is primary for our construction.

\begin{lemma} \label{supset} Suppose $f_i: M_i\to N_i$ is a $\pi_1$-surjective map of
degree $d$ between closed oriented  $3$-manifolds, $i=1,...,k$. Then
there is a $\pi_1$-surjective map $f: M_1\# ...\#M_k\to
N_1\#...\#N_k$ of degree $d$. In particular,
 $$\c{D}_{surj}(M_1\#...\#M_k, N_1\# ...\#N_k)\supset \c{D}_{surj}(M_1, N_1)\cap
 ...\cap \c{D}_{surj}(M_k, N_k).$$
\end{lemma}

\begin{proof} Suppose first $k=2$.
Since $f_{*}$ is $\pi_1$-surjective, by Lemma \ref{basic}, we can
homotopy $f_i$ such that for some $n$-ball $D'_i \subset N_i$,
$f_i^{-1}(D'_i)$ is an $n$-ball ${D}_i \subset M_i$. Thus we get a
proper map $\bar{f_i}:\ M_i\setminus D_i \rightarrow N_i\setminus
D'_i$ of degree $d$, which also induces a degree $d$ map from
$\partial D_i$ to $\partial D'_i$. Since maps of the same degree
between $(n-1)$-spheres are homotopic, so after proper homotopy, we
can paste $\bar{f}_1$ and $\bar{f}_2$ along the boundary to get map
$f=f_1\# f_2: M_1\# M_2\to N_1\# N_2$ of degree $d$. Moreover
$f_{*}=f_{1*}* f_{2*}: \pi_1{M_1}*\pi_1{M_2}\to
\pi_1{N_1}*\pi_1{N_2}$ is surjective since each $f_{i*}:
\pi_i{M_i}\to \pi_1{N_i}$ is surjective. Also clearly
$$\c{D}_{surj}(M_1\#M_2, N_1\#N_2)\supset \c{D}_{surj}(M_1, N_1)\cap
\c{D}_{surj}(M_k, N_2).$$ Then the proof of the Lemma is finished by
induction.
\end{proof}

Suppose $N=N_1\#...\#N_k$ subjects the condition in Theorem
\ref{$D(M,N)$}. To apply Lemma \ref{supset} to prove Theorem
\ref{$D(M,N)$}, for each $N_i$, we need to find a 3-manifold $M_i$
so that $\cap_{i=1}^k \c{D}_{surj}(M_i, N_i)$ is an infinite set.
The next lemma provides a uniform and the simplest way to construct
such $M_i$.

\begin{lemma}\label{n+1} Let $M$ be a closed oriented  manifold.
Suppose $M$ has a self-map  of degree $n$, i.e., $n\in \c{D}(M)$.
Then there is a $\pi_1$-surjective map $g: M\# M\to M$ of degree
$n+1$, i.e., $n+1\in \c{D}_{surj}(M\#M, M)$.
\end{lemma}

\begin{proof} Suppose $f:M\to M$ is a map of degree $n$. Pick two
copies $M_1$ and $M_2$ of $M$ and we construct the following maps
$$M_1\# M_2 \overset q  \longrightarrow M_1\vee M_2\overset {id\vee f} \longrightarrow M_1\vee M_2\overset {u} \longrightarrow M,$$
where $q$ is the quotient map which pinches the 2-sphere defining
the connected sum $M_1\# M_2$ to the point defining the one point
union $M_1\vee M_2$,  the map ${id\vee f}$ restricted on $M_1$ is
the identity and restricted on $M_2$ is the map $f$, and the map $u$
sends both $M_1$ and $M_2$ to $M$ by orientation preserving
homeomorphisms. Let $g=u \circ (id\vee f) \circ q$. Then it is easy
to see that on top dimensional homology, $g$ sends the fundamental
class $[M_1\# M_2]$ to $(n+1)[M]$ therefore $g$ of degree $n+1$.
Furthermore on the fundamental group $g_*$ sends the free factor
$\pi_1(M_1)$ of $\pi_1(M_1\# M_2)=\pi_1(M_1)*\pi_1 (M_2)$ to
$\pi_1(M)$ isomorphically, hence $g$ is $\pi_1$-surjective.
\end{proof}

According to and suggested by Lemma \ref{n+1}, we will try to find
the infinite intersection of $\c{D}(N_i\#N_i, N_i)$, and to do this
we should first find the infinite intersection of $\c{D}(N_i)$.
Lemma \ref{D(R)} below is prepared for this purpose.

To state Lemma \ref{D(R)}, we need to slightly reorganize the prime
3-manifolds $R$ with $|\c{D}(R)|=\infty$. According to Proposition
\ref{D(N)}, such $R$ is covered by either a torus bundle, or a
trivial circle bundle, or the 3-sphere $S^3$. Call a 3-manifold $R$
a torus semi-bundle if $R$ is obtained by identifying the boundaries
of  two twisted $I$-bundle over the Klein bottle. Each torus
semi-bundle is doubly covered by a torus bundle. Each 3-manifold $R$
covered by a torus bundle must be a torus bundle or a torus
semi-bundle if $R$ supports the geometry of $E^3$ or Sol. But some
3-manifolds supporting Nil geometry are neither torus bundle nor
torus semi-bundle \cite{SWWZ}. Each $R$ supporting $H^2\times
E^1$-geometry has a unique Seifert fiberation with $n$ singular
fibers of index $a_1, ..., a_n$, and we will set
$\alpha(R)=|a_1...a_n|$ if $n>0$ and $\alpha(R)=1$ if $n=0$. Now we
divide prime 3-manifolds $R$ with $|D(R)|=\infty$ into the following
five classes

(1) $R$ supports $S^3$ geometry.

(2) $R$ supports $H^2\times E^1$ geometry.

(3) $R$ is a torus bundles or torus semi-bundle;

(4) $R$ is a Nil 3-manifold not in (3);

(5) $R=S^2\times S^1$.

\begin{lemma}\label{D(R)} Suppose $R$ is a closed oriented  prime 3-manifold such that $|\c{D}(R)|=\infty$.
Then $\c{D}(R)$ contains a infinite set of integers as below:

(1) $\c{D}(R)\supset \{l|\pi_1(R)|+1| l\in \ZZ\}$ if $R$ is covered
by $S^3$;

(2) $\c{D}(R)\supset \{l\alpha(R)+1, \, l\in \ZZ\}$ if $R$ supports
$H^2\times E^1$-geometry; (3) $\c{D}(R)\supset \{(2l+1)^2| l\in
\ZZ\}$ if $R$ is a torus bundle or a torus semi-bundle;

(4) $\c{D}(R)\supset \{(l)^4| l\equiv 1\, mod\, 12, \, l\in \ZZ\}$
if $R$ supports Nil-geometry but not in Class (3).


(5) $\c{D}(R)=\ZZ$ if $R=S^2\times S^1$.
\end{lemma}

\begin{proof} (5) is obviously.
(1) and (2) are derived from known elementary
constructions, and certainly  one  can also  find (1) in \cite{W1}
\cite{Du} and  \cite{SWWZ} and  (2) in \cite{W1} and \cite{SWWZ}.

(3) is derived from  Theorem 1.6 and Theorem 1.7 of \cite{SWW}, and
(4) is derived from Theorem 1.4 of \cite{SWWZ}.
\end{proof}

We are going to prove  Theorem \ref{$D(M,N)$}. Suppose $N$ is  a
closed oriented  3-manifold and $|\c{D}(R)|=\infty$ for each prime
factor $R$ of $N$. By the discussion before Lemma \ref{D(R)}, we
have

$$N=
 (\#_{i=1}^a P_i)\# (\#_{j=1}^b Q_j)\# (\#_{k=1}^c
U_k)\# (\#_{m=1}^d V_m){\#(\#_{p=1}^e S^2\times S^1)},
$$ where $P_i$,
$Q_j$, $U_k$ and $V_m$ are 3-manifolds of types in (1), (2), (3) and
(4) respectively, and $a, b, c, d, e$ are integers $\ge 0$.
\begin{theorem}\label{$D(M,N)$*}
 Let $$d(N,l)=
(12 \prod_{i=1}^a|\pi(P_i)|\prod_{j=1}^{b}\alpha(Q_j)l+1)^4,\,\,\,
l\in \ZZ.$$

 Then $d(N,l)+1\in \c{D}_{surj}(N\# N, N)$ for each
$l\in\ZZ$.
\end{theorem}

\begin{proof} It is easy to present $d(N,l)$ in the following four forms
$$d(N,l)= C_1|\pi_1(P_i)|+1
=C_2|\alpha(Q_j)|+1=(2C_3+1)^2=(12C_4+1)^4$$ for some integers $C_1,
C_2, C_3, C_4$.

Comparing those four forms with (1), (2), (3), (4) of Lemma
\ref{D(R)} respectively, we have that $d(N,l)\in D(R)$ for each
prime factor $R$ in $N$.

By Lemma \ref{n+1}, we have that $d(N,l)+1\in \c{D}_{surj}(R\# R,
R)$ for each prime factor $R$ in $N$ and each $l\in\ZZ$.

Notice that
$$
 (\#_{i=1}^a P_i\#P_i)\# (\#_{j=1}^b Q_j\#Q_j)\# (\#_{k=1}^c
U_k\#U_k)\# (\#_{m=1}^d V_m\#V_m)\#(\#_{p=1}^{2e} S^2\times S^1
)=N\#N.$$

By Lemma \ref{supset}, we have that $d(N,l)+1\in \c{D}_{surj}(N\# N,
N)$ for each $l\in\ZZ$.

This finishes the proof of Theorem \ref{$D(M,N)$*}.
\end{proof}

Therefore we finish the proof of  Theorem \ref{$D(M,N)$}.

\vskip 0.4 truecm

\textbf{Acknowledgement.} The first two authors thank the
Mathematics Department of Peking University for support and
hospitality during their visit in summer of 2010 to work on this
paper. The third author is partially supported by grant No.10631060
of the National Natural Science Foundation of China and Ph.D. grant
No. 5171042-055 of the Ministry of Education of China.


\begin{thebibliography}{ZZZZZZ}

\bibitem[BG1]{BG1} {\sc R. Brooks, W. Goldman}, {\it The Godbillon-Vey invariant of a transversely homogeneous foliation}, Trans. Amer. Math. Soc. 286 (1984), no. 2, 651--664.


\bibitem[BG2]{BG2} {\sc R. Brooks, W. Goldman}, {\it Volumes in Seifert space},  Duke Math. J.  51  (1984),  no. 3, 529--545.



\bibitem[CT]{CT} {\sc J. Carlson, D. Toledo,} {\it Harmonic mappings of K\"ahler
manifolds to locally symmetric spaces.} Inst. Hautes \'{E}tudes Sci.
Publ. Math. No. 69 (1989), 173--201.


\bibitem[DeW]{DeW} {\sc  P. Derbez, S. C. Wang}, {\it Graph manifolds have virtualy positive Seifert volume},  math.GT (math.AT) arXiv:0909.3489.


\bibitem[Du]{Du} X.M. Du, {\it On Self-mapping Degrees of $S^3$-geometry
3-manifolds}, Acta Math. Sin. (Engl. Ser.) 25 (2009), no. 8,
1243--1252.

\bibitem[DW]{DW} H.B. Duan; S.C. Wang,  {\it Non-zero degree maps
between $2n$-manifolds.} Acta Math. Sin. (Engl. Ser.) 20 (2004), no.
1, 1--14.


\bibitem[G]{G} {\sc M. Gromov}, {\it Volume and bounded cohomology},  Inst. Hautes \'{E}tudes Sci. Publ. Math. No. 56, (1982), 5--99.



\bibitem[MT]{MT} {\sc J. Milnor, W. Thurston,} {\it  Characteristic numbers of
$3$-manifolds.} Enseignement Math. (2) 23 (1977), no. 3-4, 249--254.

\bibitem[Re]{Re} {\sc A. Reznikov},
{\it Volumes of discrete groups and topological complexity of
homology spheres}, Math. Ann. 306 (1996), no. 3, 547--554.


\bibitem[RW]{RW} Y.W. Rong and S.C. Wang, \emph{The preimage of submanifolds},
Math. Proc. Camb. Phil. Soc. \textbf{112}, 1992, 271-279.

\bibitem[Sc]{Sc}
{P. Scott},
\newblock {\em The geometries of 3-manifolds},
\newblock Bull. Lond. Math. Soc. 15 (1983), 401-487.

\bibitem[So]{So} T. Soma,  \emph{The Gromov invariant of links.} Invent. Math.
\textbf{64} 1981, 445--454.


\bibitem[St]{St} J. Stallings, {\em A topological proof of Grushko's theorem on free
products.} Math. Z. 90 1965 1--8.




\bibitem[SWW]{SWW} H. B. Sun, S. C. Wang, J. C. Wu, {\it Self-mapping degrees of
torus bundles and torus semi-bundles} Osaka J. Math. 47 (2010), no.
1, 131-155.

\bibitem[SWWZ]{SWWZ} H. B. Sun, S. C. Wang, J. C. Wu, H. Zheng, {\it Self-mapping
Degrees of 3-Manifolds.}  math.GT (math.AT) arXiv:0810.1801.


\bibitem[Th]{Th} {\sc W. Thurston},  {\it Three-dimensional manifolds, Kleinian groups and
hyperbolic geometry}. Bull. Am. Math. Soc. 6, 357--381 (1982)


\bibitem[W1]{W1} {\sc S.C. Wang}, {\it The $\pi\sb 1$-injectivity of self-maps of nonzero degree on $3$-manifolds},  Math. Ann.  297  (1993),  no. 1, 171--189.

\bibitem[W2]{W2} {\sc S.C. Wang}, {\it Non-zero degree maps between 3-manifolds},   Proceedings of the International Congress of Mathematicians, Vol. II (Beijing, 2002),  457--468, Higher Ed. Press, Beijing, 2002.




\end{thebibliography}
\end{document}